\documentclass[a4paper,10pt]{amsart}
\usepackage[a4paper,margin=3cm]{geometry}
\usepackage[T1]{fontenc}
\usepackage{amsfonts,amsthm,amssymb,amsmath}
\usepackage{amssymb}
\usepackage[pdftex]{color,graphicx}
\numberwithin{equation}{section}
\title{The strong asymptotic freeness of Haar and deterministic matrices}
\author{B. Collins}
\address{
D\'epartement de Math\'ematique et Statistique, Universit\'e d'Ottawa,
585 King Edward, Ottawa, ON, K1N6N5 Canada
and 
CNRS, Institut Camille Jordan Universit\'e  Lyon 1, 43 Bd du 11 Novembre 1918, 69622 Villeurbanne, 
France} 
\email{bcollins@uottawa.ca}
\author{C. Male}
\address{\'Ecole Normale Sup\'erieure de Lyon, Unit\'e de Math\'ematiques Pures et Appliqu\'ees, UMR CNRS 5669, 
46 all\'ee d'Italie, 
69364 LYON Cedex 07, 
FRANCE} 
\email{camille.male@ens-lyon.fr}
\date{}

\newtheorem{Th}{Theorem}[section]

\newtheorem{Def}[Th]{Definition}
\newtheorem{Prop}[Th]{Proposition}

\newtheorem{Lem}[Th]{Lemma}

\newtheorem{Cor}[Th]{Corollary}

\renewcommand\leq\leqslant

\renewcommand\geq\geqslant
\def\eps{\varepsilon}
\def\Tr{\mathrm{Tr}}

\def\esp{\mathbb E}

\def\etc{,\ldots ,}

\def\Mk{\mathrm{M}_k(\mathbb C)  }
\def\MN{\mathrm{M}_N(\mathbb C) }

\def\toN{^{(N)}}
\def\toNs{^{(N)*}}

\def\toN{^{(N)}}

\def\limN{\underset{N \rightarrow \infty}\longrightarrow}

\begin{document}
\maketitle
\begin{abstract}
In this paper, we are interested in sequences of 
$q$-tuple of $N\times N$ 
random matrices having a strong limiting distribution
(i.e. given any non-commutative polynomial in the matrices and their conjugate transpose, 
its normalized trace and its norm converge).
We start with such a sequence having this property, and we show that this property pertains if
the 
$q$-tuple 
is enlarged with independent unitary Haar distributed random matrices. Besides, the limit of norms and traces in 
non-commutative polynomials in the enlarged family
can be computed with reduced free product construction.
This extends
results of one author (C. M.) and of Haagerup and Thorbj\o rnsen.
We also show that a $p$-tuple of independent orthogonal and symplectic Haar matrices have a strong limiting distribution, extending a recent result of Schultz.
We mention a couple of applications in random matrix and operator space theory.

\emph{Dans cet article, nous nous int\'eressons au $q$-tuple de matrices $N\times N$ qui ont une distribution
limite forte (i.e. pour tout polyn\^ome non-commutatif en les matrices et leurs adjoints, sa trace
normalis\'ee et sa norme convergent).
Nous partons d'une telle suite de matrices al\'eatoires et montrons que cette propri\'et\'e persiste
si on rajoute au $q$-tuple des matrices ind\'epandantes unitaires distribu\'des suivant la
mesure de Haar. Par ailleurs, la limite des normes et des traces en des polyn\^omes non-commutatifs
en la suite \'elargie peut \^etre calcul\'e avec la construction du produit libre r\'eduit.
Ceci \'etend les r\'esultats d'un des auteurs (C.M.) et de Haagerup et Thorbj\o rnsen.
Nous montrons aussi qu'un $p$-tuple de matrices ind\'epandantes orthogonales et symplectiques 
ont une distribution limite forte, \'etendant par l\`a-m\^eme un r\'esultat de Schultz. 
Nous passons aussi en revue quelques applications de notre r\'esultat aux matrices al\'eatoires
et \`a la th\'eorie des espaces d'op\'orateur.}
\end{abstract}

\section{Introduction and statement of the main results}
\noindent Following random matrix notation, we call GUE the Gaussian Unitary Ensemble, i.e. any sequence $(X_{N})_{N\geq 1}$
of random variables where $X_{N}$ is an $N\times N$ selfadjoint random matrix whose distribution is proportional to
the measure
 $\exp \big( -N/2 \Tr (A^{2}) \big)\mathrm d{A}$, 
where $\mathrm{d}A$ denotes the Lebesgue measure on the set of $N \times N$ Hermitian matrices. We call a unitary Haar matrix of size $N$ any random matrix distributed according to the Haar measure on the compact group of $N$ by $N$ unitary matrices.
\\
\\We recall for readers' convenience the following definitions 
from free probability theory (see \cite{AGZ,NS}).

\begin{Def} ~

\begin{enumerate}
	\item A {\bf$\mathcal C^*$-probability space} $(\mathcal A, .^*, \tau ,  \|\cdot \|)$ consists of a unital $C^*$-algebra $(\mathcal A,.^*, \| \cdot \|)$ endowed with a state $\tau$, i.e. a linear map $\tau \colon\mathcal A\to\mathbb C$
satisfying $\tau[\mathbf 1_\mathcal A]=1$ and $\tau[aa^{*}]\geq 0$ for all $a$ in $\mathcal A$. 
In this paper, we always assume that $\tau$ is a trace, i.e. that it satisfies $\tau[ab]=\tau[ba]$ for every $a,b$ in $\mathcal A$. An element of $\mathcal A$ is called
a (non commutative) random variable. 
A trace is said to be {\bf faithful} if $\tau[aa^*]>0$ whenever $a\neq 0$. If $\tau$ is faithful, then for any $a$ in $\mathcal A$,
		\begin{equation} \label{NormTrace}	
			\| a \| = \lim_{k\rightarrow \infty} \Big( \tau\big[ (a^*a)^k \big] \Big).
		\end{equation}
	\item Let $\mathcal A_1, \ldots ,\mathcal A_k$ be $*$-subalgebras of $\mathcal A$ having the same unit as $\mathcal A$.
They are said to be {\bf free} if for all $a_i\in \mathcal  A_{j_i}$ ($i=1, \ldots, k$, $j_i\in \{1\etc k\}$) 
such that $\tau[a_i]=0$, one has  
		$$\tau[a_1\cdots a_k]=0$$
as soon as $j_1\neq j_2$, $j_2\neq j_3,\ldots ,j_{k-1}\neq j_k$.
Collections of random variables are said to be 
free if the unital subalgebras they generate are free.

	\item Let $\mathbf a= (a_1,\ldots ,a_k)$ be a $k$-tuple of random variables. The {\it joint distribution\it} of the family $\mathbf a$ is the linear form 
$P \mapsto \tau\big[ P(\mathbf a, \mathbf a^*) \big]$ on the set of polynomials in $2k$ non commutative indeterminates. 
By {\bf convergence in distribution}, for a sequence of families of variables $(\mathbf a_N)_{N\geq 1} = (a_{1}^{(N)},\ldots ,a_{k}^{(N)})_{N\geq 1}$ in $\mathcal C^*$-algebras $\big( \mathcal A_N, .^*, \tau_N, \| \cdot \| \big)$,
we mean the pointwise convergence of
the map 
$$P \mapsto \tau_N \big[ P(\mathbf a_N, \mathbf a_N^*) \big],$$
and by {\bf strong convergence in distribution}, we mean convergence in distribution, and pointwise convergence
 of the map
$$P \mapsto \big\| P(\mathbf a_N, \mathbf a_N^*) \big\|.$$

	\item A family of non commutative random variables $\mathbf x=(x_1 \etc x_p)$ is called 
	a {\bf free semicircular system} when the non commutative random variables are free, 
	selfadjoint ($x_i=x_i^*$, $i=1 \etc p$), and for all $k$ in $\mathbb N$ and $i=1 , \ldots , p$, one has
\begin{equation*}
		\tau[ x_i^k] =  \int t^k d\sigma(t),
\end{equation*}
with $d\sigma(t) = \frac 1 {2\pi} \sqrt{4-t^2} \ \mathbf 1_{|t|\leq2} \ dt$ the semicircle distribution.

	\item A non commutative random variable $u$ is called a {\bf Haar unitary} when it is unitary ($uu^*=u^*u=\mathbf 1_{\mathcal A}$) and for all $n$ in $\mathbb N$, one has
\begin{equation*}
		 \tau[ u^n]  =\left \{ \begin{array}{cc}   1 & \mathrm{if} \  n=0,\\
						0 & \mathrm{otherwise}. \end{array}\right.
\end{equation*}
\end{enumerate}
\end{Def}
~\\
In their seminal paper \cite{HT}, Haagerup and Thorbj\o rnsen 
proved the following result. 

\begin{Th}[ \cite{HT} The strong asymptotic freeness of independent GUE matrices]\label{thm:HT} ~\\
For any integer $N\geq 1$, let $X_1\toN \etc X_p\toN$ be $N\times N$ independent GUE matrices and let $(x_1 \etc x_p)$ be a free semicircular system in a $\mathcal C^*$-probability space with faithful state. Then, almost surely, for all polynomials $P$ in $p$ non commutative indeterminates, one has
		$$\big\| P(X_1\toN \etc X_p\toN) \big\| \limN \big\| P(x_1 \etc x_p) \big\|,$$
where $\| \cdot \|$ denotes the operator norm in the left hand side and the norm of the 
$\mathcal C^*$-algebra in the right hand side.
\end{Th}

\noindent This theorem is a very deep result in random matrix theory, and had an important impact.
Firstly, it had significant applications to $C^{*}$-algebra theory \cite{HT,PIS}, and more recently to 
quantum information theory \cite{BCN,CN}. Secondly, it was generalized in many directions. 
Schultz \cite{SCH} has shown that Theorem \ref{thm:HT} is true when the GUE matrices are 
replaced by matrices of the Gaussian Orthogonal Ensemble (GOE) or by matrices of the 
Gaussian Symplectic Ensemble (GSE). Capitaine and Donati-Martin \cite{CD} and, 
very recently, Anderson \cite{AND} have shown the analogue for certain Wigner matrices.
\\
\\An other significant extension of Haagerup and Thorbj\o rnsen's result was
obtained by one author (C. M.) in \cite{MAL},
where he 
showed
that if in addition to 
independent GUE matrices, one also has an extra family of 
independent matrices with strong limiting distribution, 
the result still holds.

\begin{Th}[ \cite{MAL} The strong asymptotic freeness of $\mathbf X_N,\mathbf Y_N$]\label{thm:CM} ~\\
For any integer $N\geq 1$, we consider
\begin{itemize}
	\item a $p$-tuple $\mathbf X_N$ of $N \times N$ independent GUE matrices,
	\item a $q$-tuple $\mathbf Y_N$ of $N\times N$ matrices, possibly random but independent of $\mathbf X_N$.
\end{itemize}
The above random matrices live in the 
$\mathcal C^*$-probability space $(\MN, .^*, \tau_N, \| \cdot \|)$,
where $\tau_N$ is the normalized trace on the set $\MN$ of $N\times N$ matrices. In a $\mathcal C^*$-probability space $(\mathcal A, .^*, \tau, \| \cdot \|)$ with faithful trace, we consider
\begin{itemize}
	\item a free semicircular system $\mathbf x$ of $p$ variables,
	\item a $q$-tuple $\mathbf y$ of non commutative random variables, free from $\mathbf x$.
\end{itemize}

\noindent If $\mathbf y$ is the strong limit in distribution of $\mathbf Y_N$, then $(\mathbf x, \mathbf y)$ is the strong limit in distribution of $(\mathbf X_N, \mathbf Y_N)$.

\end{Th}

\noindent In other words, if we assume that almost surely, for all polynomials $P$ in $2q$ non commutative indeterminates, one has
\begin{eqnarray}
		\tau_N\big [ P(\mathbf Y_N, \mathbf Y_N^* )  \big ] & \underset{N\rightarrow \infty}{\longrightarrow} & \tau \big [ P(\mathbf y, \mathbf y^*)  \big],\label{Th1wCV} \\
		\big \| P(\mathbf Y_N, \mathbf Y_N^*)   \big \| & \underset{N\rightarrow \infty}{\longrightarrow} & \big \| P(\mathbf y, \mathbf y^*)  \big \|,\label{Th1sCV}
\end{eqnarray}
then, almost surely, for all polynomials $P$ in $p+2q$ non commutative indeterminates, one has
\begin{eqnarray}
		\tau_N\big [ P( \mathbf X_N,  \mathbf Y_N, \mathbf Y_N^* )  \big ] & \underset{N\rightarrow \infty}{\longrightarrow} & \tau \big [ P(\mathbf x,\mathbf y, \mathbf y^*) \big  ],\label{MainThEqVoic} \\
		\big  \| P( \mathbf X_N,  \mathbf Y_N,  \mathbf Y_N^*)  \big \| & \underset{N\rightarrow \infty}{\longrightarrow} & \big \| P( \mathbf x, \mathbf y,\mathbf y^*) \big \| \label{MainThEq}.
\end{eqnarray}

\noindent The convergence in distribution, stated in (\ref{MainThEqVoic}), is the content of Voiculescu's asymptotic freeness theorem.
We refer to \cite[Theorem 5.4.10]{AGZ}
for the original statement and for a proof. 
An alternative way to state \eqref{MainThEq} is the following interversion of limits: for any matrix $H_N = P( \mathbf X_N,  \mathbf Y_N, \mathbf Y_N^*)$, where $P$ is a fixed polynomial, if we denote $h=P( \mathbf x, \mathbf y, \mathbf y^*)$, then by the definition of the norm in terms of the state \eqref{NormTrace},
	\begin{eqnarray*}
		\lim_{N \rightarrow \infty}  \lim_{k \rightarrow \infty} \bigg( \tau_N\big [ (H_N^* H_N)^k  \big ] \bigg)^{\frac 1 {2k}} & =& \lim_{k \rightarrow \infty} \bigg( \tau\big [ (h^* h)^k  \big ] \bigg)^{\frac 1 {2k}}.
	\end{eqnarray*}

\noindent It is natural to wonder whether, instead of GUE matrices,  
the same property holds for unitary Haar matrices.
The main result of this paper is the following theorem.

\begin{Th}[The strong asymptotic freeness of $U_1\toN \etc U_p\toN,\mathbf Y_N$]\label{MainTh} ~\\
For any integer $N\geq 1$, we consider
\begin{itemize}
	\item a $p$-tuple $\mathbf U_N$ of 
$N \times N$ independent unitary Haar matrices,
	\item a $q$-tuple $\mathbf Y_N$ of $N\times N$ matrices, possibly random but independent of $\mathbf U_N$.
\end{itemize}
In a $\mathcal C^*$-probability space $(\mathcal A, .^*, \tau, \| \cdot \|)$ with faithful trace, we consider
\begin{itemize}
	\item a $p$-tuple $\mathbf u$ of free Haar unitaries,
	\item a $q$-tuple $\mathbf y$ of non commutative random variables, free from $\mathbf u$.
\end{itemize}
If $\mathbf y$ is the strong limit in distribution of $\mathbf Y_N$, then 
$(\mathbf u, \mathbf y)$ is the strong limit in distribution of $(\mathbf U_N, \mathbf Y_N)$.
\end{Th}
\noindent 

\noindent In order to solve this problem, it looks at first sight natural to attempt to mimic the proof of 
\cite{HT} and write a Master equation in the case of unitary 
matrices. However, even though such an identity can be obtained for unitary matrices, 
it is very difficult to manipulate it in the spirit of \cite{HT} in order to obtain the desired norm
convergence. Part of the problem is that the unitary matrices are not selfadjoint, unlike the GUE matrices
considered in \cite{HT}, and in this context the linearization trick and the identities 
do not seem to fit well together. In order to bypass this problem, in this paper, we take a completely different route by building on 
Theorem \ref{thm:CM} and using a series of
folklore facts of classical probability and random matrix theory.
\\
\\Our method applies to prove the strong convergence in distribution of Haar matrices on the orthogonal 
and the symplectic groups by building on the result of Schultz \cite{SCH}, which is the analogue of 
Theorem \ref{thm:HT} for GOE or GSE matrices instead of GUE matrices. The analogue of 
Theorem \ref{thm:CM} does not exist yet. 
If one shows that the estimates of matrix valued Stieltjes transforms in \cite{MAL} can always be 
performed with the additional terms in the estimate of \cite{SCH}, then, following the lines of this paper, 
one gets Theorem \ref{thm:CM} for Haar matrices on the orthogonal and the symplectic groups, instead 
of the unitary group only. 
Therefore, in the following Theorem, we stick to proving 
the strong convergence of independent unitary, orthogonal or symplectic Haar matrices,
without ``constant'' matrices $\mathbf Y$:

\begin{Th}[The strong asymptotic freeness of independent Haar matrices]\label{thm:COE} ~\\
For any integer $N\geq 1$, let $U_1\toN \etc U_p\toN$ be a family of independent Haar matrices of one of the three classical groups.
Let $u_1 \etc u_p$ 
be free Haar unitaries in a $\mathcal C^*$-probability space with faithful state. 
Then, almost surely, for all polynomials $P$ in $2p$ non commutative indeterminates, one has
		$$\big \|P(U_1\toN \etc U_p\toN, U_1\toNs \etc U_p\toNs) \big \| \limN \big \|P(u_1 \etc u_p,u_1^* \etc u_p^*) \big \|,$$
where $\| \cdot \|$ denotes the operator norm in the left hand side and the $\mathcal C^*$-algebra in 
the right hand side.
\end{Th}

\noindent Our paper is organized as follows.
Section \ref{sec:application} consists
of applications of the results stated above.
Among other examples, we show that the limit of 
complicated 
random matrix models involving 
unitary random matrices have norms that converge towards 
(or are bounded by) 
values predicted by the theory of free probability.
Sections \ref{sec:proof} and \ref{proofCOE} provide the proofs of Theorem \ref{MainTh} and Theorem \ref{thm:COE} respectively. 
Section \ref{sec:ProofCorSumProd} is dedicated to the proof of Corollary \ref{cor:SumProd}, stated in the next section.

\section{Applications}
\label{sec:application}

\noindent Our main result has the potential for many applications.

\subsection{The spectrum of Hermitian random matrices}

\subsubsection{Generalities on the strong convergence in distribution}

\noindent We first recall for convenience some facts about the strong convergence in distribution, mainly an equivalent formulation. Given a self-adjoint variable $h$ in a $\mathcal C^*$-probability space $(\mathcal A, .^*, \tau, \| \cdot \|)$, its spectral distribution $\mu_h$ is the unique probability measure that satisfies $\tau[h^k] = \int t^k \textrm d \mu(t)$ for any $k\geq 1$. This measure has compact support included in $\big [-\|h\|, \|h\| \big]$.
For any continuous map $f : \mathbb R$ to $\mathbb C$, the variable $f(h)$ is given by functional calculus, and coincides with the limit of $\big( P_n(h) \big)_{n\geq 1}$ in $\mathcal A$, where $( P_n)_{n\geq 1}$ is any Weierstrass's approximation of $f$ by polynomials. Given a (non self-adjoint) variable $x$ in $\mathcal A$, we set the self-adjoint variables $\Re x = \Big( \frac {x + x^*}2 \Big)$ and $\Im x = \Big( \frac {x - x^*}{2i} \Big)$, so that $x = \Re x + i \Im x$.
\\
\begin{Prop}[The strong convergence in distribution of self adjoint random variables]~\label{Prop:ConvEqui}
\\Let $\mathbf x_N = (x_1\toN \etc x_p\toN)$ and $\mathbf x = (x_1  \etc x_p )$ be  $p$-tuples of variables in $\mathcal C^*$-probability spaces, \break $(\mathcal A_N, .^*, \tau_N, \| \cdot \|)$ and $(\mathcal A, .^*, \tau, \| \cdot \|)$, with faithful states. Then, the following assertions are equivalent.
\begin{enumerate}
	\item $\mathbf x_N$ converges strongly in distribution to $\mathbf x$,
	\item for any continuous map $f_i, g_i: \mathbb R \to \mathbb C$, $i=1\etc p$, the family of 
	variables \break $\big( f_1( \Re x_1\toN), g_1( \Im x_1\toN) \etc f_p( \Re x_p\toN), g_p( \Im x_p\toN) \big)$ converges strongly in distribution to \break $\big( f_1( \Re x_1), g_1( \Im x_1) \etc f_p( \Re x_p), g_p( \Im x_p) \big)$,
	\item for any self-adjoint variable $h_N=P(\mathbf x_N)$, where $P$ is a fixed polynomial, $\mu_{h_N}$ converges in weak-$*$ topology to $\mu_h$ where $h=P( \mathbf x)$. Weak-$*$ topology means relatively to continuous functions on $\mathbb C$. Moreover, the support of $\mu_{h_N}$ converges in Hausdorff distance to the support of $\mu_{h}$, that is: for any $\eps >0$, there exists $N_0$ such that for any $N\geq N_0$, 
	\begin{equation}
		\mathrm{Supp }\big( \ \mu_{h_N} \ \big) \ \subset \ \mathrm{Supp }\big( \ \mu_{h} \ \big) \ + (-\eps,\eps).
	\end{equation}
	The symbol $\mathrm{Supp }$ means the support of the measure.
\end{enumerate}

\noindent In particular, the strong convergence in distribution of a single self-adjoint variable is its convergence in distribution together with the Hausdorff convergence of its spectrum.

\end{Prop}

\begin{proof} Assuming (1), the assertion (2) is obtained by Weierstrass's approximation of the functions $f_i$ and $g_i$ by polynomials in $p$ complex variables on the centered ball of radius $\sup_{N\geq 0} \| x_N \|$. The 
converse 
is obvious.
\\
\\Assuming (1), let us show (3). By Weierstrass's approximation, $h_N$ converges strongly in distribution to $h$. The convergence in distribution of $h_N$ to $h$ implies the weak-$*$ convergence of $\mu_{h_N}$ to $\mu_h$. For any $\eps >0$, let $f_\eps$ be a continuous map which takes the value 1 on the complementary of $\textrm{Supp } ( \mu_{h} )   + (-\eps,\eps)$ and 0 on $\textrm{Supp }( \mu_{h})$. Then, $\| f_\eps( \mathbf x_N) \| \limN \| f_\eps( \mathbf x) \| = \lim_k ( \int f_\eps( \mathbf x)^k \textrm d\mu_h )^{\frac 1 k}=0$. Hence, the support of $\mu_{h_N}$ is a subset of $\textrm{Supp } ( \mu_{h} ) + (-\eps,\eps)$ for $N$ large enough, as otherwise one could have $\|f_\eps(\mathbf x_N)\|=1$ eventually.
\\
\\Assuming (3), let us show (1). Let $P$ be a polynomial in $p$ variables and their conjugate. Denote $m_N = P( \mathbf x_N, \mathbf x_N^*)$ and $m = P( \mathbf x, \mathbf x^*)$. Then, $\tau_N[m_N] - \tau[m ] = \tau_N[Q(\mathbf x_N,\mathbf x_N^*)] - \tau[Q(\mathbf x,\mathbf x^*)] + i (\tau_N[R(\mathbf x_N,\mathbf x_N^*)] - \tau[R(\mathbf x,\mathbf x^*)]  )$ where $Q=\frac 1 2 (P+P^*)$ and $R=\frac 1 {2i}(R-R^*)$ gives Hermitian variables. By the assertion (3) and since the matrices are uniformly bounded in operator norm, we get the convergence in moments of the spectral distribution of $Q(\mathbf x_N,\mathbf x_N^*)$ and $R(\mathbf x_N,\mathbf x_N^*)$. Hence, we get the convergence in distribution of $x_N$ to $x$. Then, the convergence holds in weak-$*$ topology since $\mu_h$ has bounded support. Furthermore, 
	$$\|m_N\|^2 = \|m_N^* m_N\| = \max \textrm{Supp } \big( \ \mu_{m_N^* m_N} \ \big) \limN  \max \textrm{Supp } \big ( \ \mu_{m^* m} \ \big) = \|m^*m\| =  \|m\|^2.$$
\end{proof}

\subsubsection{The spectra of the sum and product of unitary invariant matrices}

The following is a consequence of our main result:

\begin{Cor}\label{cor:SumProd} Let $A_N,B_N$ be two $N \times N$ independent Hermitian random matrices. Assume that:
\begin{enumerate}
	\item the law of one of the matrices is invariant under unitary conjugacy,
	\item almost surely, the empirical eigenvalue distribution of $A_N$ (respectively $B_N$) converges to a compactly supported probability measure $\mu$ (respectively $\nu$),
	\item almost surely,
	for any neighborhood of the support of $\mu$ (respectively $\nu$),
	 for $N$ large enough, the eigenvalues of $A_N$ (respectively $B_N$) belong to 
	 the respective neighborhood.
\end{enumerate}
Then, one has
\begin{itemize}
	\item almost surely, for $N$ large enough, the eigenvalues of $A_N+B_N$ belong to a small neighborhood of the support of $\mu \boxplus \nu$, where $\boxplus$ denotes the free
	additive
	 convolution (see \cite[Lecture 12]{NS}).
	\item if moreover $B_N$ is nonnegative, then the eigenvalues of $(B_N)^{1/2}A_N(B_N)^{1/2}$ belong to a small neighborhood of the support of $\mu \boxtimes \nu$, where $\boxtimes$ denotes the free 
	multiplicative
	convolution (see \cite[Lecture 14]{NS}).
\end{itemize}
\end{Cor}

\noindent Corollary \ref{cor:SumProd} is proved in Section \ref{sec:ProofCorSumProd}. It can be applied in the following situation. 
Let $A_{N}$ be an $N\times N$ Hermitian random matrix whose law is invariant under unitary conjugacy. 
Assume that, almost surely, the empirical eigenvalue distribution of $A_N$ converges to a 
compactly supported 
probability measure $\mu$ and its eigenvalues belong to the support of $\mu$ for $N$ large enough. 
Let $\Pi_N$ be the matrix of the projection on first $p_N$ coordinates, 
$\Pi_N = \mathrm{diag} \ (\mathbf 1_{p_N}, \mathbf 0_{N-p_N}),$ 
where $p_{N}\sim tN$, $t\in (0,1)$. 
We consider the empirical eigenvalue distribution $\mu_N$ of the Hermitian random matrix
		$$\Pi_n A_n \Pi_n.$$
Then, it follows from a Theorem of Voiculescu \cite{VOI4} (see also \cite{COL}) that almost surely $\mu_ N$ converges weakly to the probability measure 
$\mu^{(t)}=\mu\boxtimes [(1-t)\delta_{0}+t\delta_{1}]$. This distribution 
is important in free probability theory because of its close relationship to the free additive convolution semigroup (see \cite[Exercise 14.21]{NS}). Besides, the empirical eigenvalue distribution $\mu_{N}$ 
was proved to be
a determinantal point process obtained as the push forward of a uniform measure in a 
Gelfand-Cetlin cone \cite{DEF}.
Very recently, it was proved by Metcalfe \cite{MET} that the eigenvalues satisfy universality 
property inside the bulk of the spectrum. 
Our result complement his, by showing that almost surely, for $N$ large enough there is no 
eigenvalue outside of any neighborhood of the 
spectrum of $\mu^{(t)}$.

\subsection{Questions from operator space theory}

\noindent We present some examples of norms of large matrices we can compute by Theorem \ref{MainTh}, as the norm of the limiting variables have been computed by other authors.

\subsubsection{The norm of the sum of unitary Haar matrices}
The following question was raised by Gilles Pisier to one author (B.C.) ten years ago:
let $U_{1}\toN \etc U_{p}\toN$ be $N\times N$ independent unitary Haar random matrices, $p\geq 2$. Is it true that almost surely:
\begin{equation}\label{eq:QuestionPisier}
		 \Big \| \sum_{i=1}^{p}U_{i}\toN \Big\| \limN 2\sqrt{p-1}.
\end{equation}

\noindent This question is very natural from the operator space theory 
point of view \cite[Chapter 20]{PIS}, and was still open before this paper. 
Haagerup and Thorbj\o rnsen's theorem \cite{HT} have proved that the 
the convergence \eqref{eq:QuestionPisier} is true 
when $U_{1}\toN \etc U_{p}\toN$ are 
certain sequence of independent large unitary matrices (non Haar distributed). 
Our main theorem implies that \eqref{eq:QuestionPisier} is true almost surely 
when they are i.i.d. unitary Haar matrices. 
Indeed, $2\sqrt{p-1}$ is the norm of 
the sum of $p$ free Haar unitaries by a result of Akemann and 
Ostrand \cite{AO}: they have proved that if $u_i$ are the generators of the free group von Neumann algebra,
then 
	\begin{equation} \label{AkOs}
		 \Big\| \sum_{i=1}^{p} a_i u_{i} \Big \| = \min_{t\geq 0}\Big\{ 2t + \sum_{i=1}^p \big( \sqrt{t^2 + |a_i|^2} -t \big) \Big\}.
	\end{equation}
 And if $a_1=\ldots =a_p=1$ they prove that the minimizer of the right hand side is $2\sqrt{p-1}$.
 \\
 \\By Theorem \ref{thm:COE} and \eqref{AkOs}, we get that, for independent Haar matrices $U_1\toN \etc U_p\toN$ on the orthogonal, unitary or symplectic group,
almost surely one has
		$$ \Big\| \sum_{i=1}^{p} a_i U_{i}\toN \Big \|  \limN \min_{t\geq 0}\Big\{ 2t + \sum_{i=1}^p \big( \sqrt{t^2 + |a_i|^2} -t \big) \Big\},$$
		which is a generalization of (\ref{eq:QuestionPisier}).

\subsubsection{The sum of Haar matrices along with their conjugate} In the same vein, by a result of Kesten \cite{KES}, the norm of the sum of $p$ free Haar unitaries 
and of their conjugate equals $2\sqrt{2p-1}$. Hence, we get from our result that almost surely one has
		$$\Big \| \sum_{i=1}^{p} \big( U_{i}\toN + U_i\toNs \big) \Big\| \limN 2\sqrt{2p-1}.$$
\noindent Remark that this result is not true for random unitary matrices distributed according to the uniform measure on the set of permutation matrix. Indeed, in that case $2p$ is always an eigenvalue of the matrix since $ \sum_{i} ( U_{i}\toN + U_i\toNs)$ is the adjacency matrix of a $2p$-regular graph. The convergence of the second largest eigenvalue to $2\sqrt{2p-1}$, known as Alon's conjecture \cite{ALO86}, has been proved recently by Friedman \cite{FRI03}.

\subsubsection{The sum of Haar matrices, matrix valued case}
Lehner \cite{LEH} 
has proved that for $u_1 \etc u_p$ free Haar unitaries and $a_0, a_1 \etc a _p$ Hermitian $k$ by $k$ matrices
	\begin{equation} \label{eq:Lehner}
		\bigg \| a_0 \otimes \mathbf 1 + \sum_{i=1}^p a_i \otimes u_i \bigg\|   =   \inf_{b >0} \bigg\| b^{\frac 1 2} \Big(  \big( \mathbf 1_k + ( b^{-\frac 1 2} a_i b^{-\frac 1 2})^2 \big)^{\frac 1 2} - \mathbf 1_k \Big) b^{\frac 1 2} \bigg\|,
	\end{equation}

\noindent where the infimum is over all positive definite invertible $k$ by $k$ matrices $b$. Recall that from Theorem \ref{thm:COE} we can deduce 
the following corollary (see \cite[Proposition 7.3]{MAL} for a proof).
\begin{Cor}\label{cor:BlockMatrices} Let $\mathbf U_N$ be a family of independent Haar matrices one of the three classical groups. Let $\mathbf u$ be a family of free Haar unitaries. Let $k\geq 1$ be an integer. Then, for any polynomial $P$ with coefficients in $\Mk$, almost surely one has
		$$\| P(\mathbf U_N, \mathbf U_N^*) \| \limN \| P(\mathbf u, \mathbf u^*) \|,$$
where $\| \cdot \|$ stands in the left hand side for the operator norm in $\mathrm{M}_{kN}(\mathbb C)$ 
and in the right hand side for the $\mathcal C^*$-algebra norm in $\mathrm{M}_{k}(\mathcal A)$.
\end{Cor}

\noindent We then deduce that the norm of block matrices of the form $ a_0 \otimes \mathbf 1 + \sum_{i=1}^p a_i \otimes U_i\toN$, where $a_0 \etc a_1$ are Hermitian matrices, converges almost surely to the quantity \eqref{eq:Lehner} computed by Lehner.

\subsubsection{Application of Fell's absorption principle} For an other application of Corollary \ref{cor:BlockMatrices}, recall Fell's absorption principle \cite[Proposition 8.1]{PIS}: for any $k$ by $k$ unitary matrices $a_1 \etc a_p $ and $u_1Ê\etc u_p$ a free Haar unitaries, one has
	\begin{equation*}
		\Big \| \sum_{i=1}^{p}  a_i \otimes u_i \Big\| = \Big \| \sum_{i=1}^{p}  u_i \Big\| = 2 \sqrt{p-1}.
	\end{equation*}

\noindent By Corollary \ref{cor:BlockMatrices} we get for any $k \times k$ unitary 
matrices $a_1 \etc a_p$, almost surely one has
		 $$\Big \| \sum_{i=1}^{p}a_iÊ\otimes U_{i}\toN \Big\| \limN 2\sqrt{p-1},$$

\noindent which solves a question of Pisier in \cite[Chapter 20]{PIS}.

\subsection{Estimates on the norm of random matrices}

\subsubsection{Haagerup's inequalities} Let $\mathbf u=(u_1 \etc u_p)$ be free Haar unitaries 
in a $\mathcal C^*$-probability space $(\mathcal A, .^*, \tau, \| \cdot \| )$ with faithful state. 
For any integer $d\geq 1$, we denote by $W_d$ the set of reduced $^*$-monomials in $p$ indeterminates $\mathbf x = (x_1 \etc x_p)$ of length $d$:
\begin{eqnarray*}
		 W_d &=  & \Big\{ \ P= x_{j_1}^{\varepsilon_1} \dots  x_{j_d}^{\varepsilon_d} \ \Big | \ j_1 \neq \dots \neq j_d, \ \varepsilon_j \in \{1, *\}  \ \forall j=1\etc d \ \Big \}.
\end{eqnarray*}
In 1979, Haagerup \cite{HAA} has shown that one has 
\begin{equation}\label{eq:HaaIneq}
		\Big \| \sum_{n\geq 1} \alpha_n P_n(\mathbf u)  \Big \| \leq (d+1) \| \alpha\|_2,
\end{equation}
for any sequence $(P_n)_{n\geq 1} $ of elements in $W_d$ and sequence 
$\alpha =(\alpha_n)_{n \geq1 }$ of complex numbers whose $\ell^2$-norm is denoted by
		$$\| \alpha \|_2 = \sqrt{ \sum_{n\geq 1} |\alpha|^2}.$$
This result, known as Haagerup's inequality, has many applications (for example, estimates
of return probabilities of random walks on free groups) 
and has been generalized in many ways. 
For instance, Buchholz has generalized (\ref{eq:HaaIneq}) in an estimate of 
$\sum_{n\geq 1} a_n \otimes  x_n$, where the $a_n$ are now $k\times k$ matrices. 
Let $\mathbf U_N$ be a family of $p$ independent $N \times N$ unitary Haar matrices. 
As a byproduct of our main result, we get that
$$ \lim_{N \rightarrow \infty} \| \sum_{n\geq 1} \alpha_n P_n(\mathbf U_N) \| \leq (d+1) \| \alpha \|_2,$$
where for any $n\geq 1$, the polynomial $P_n$ is in $W_d$.

\subsubsection{Kemp and Speicher's inequality} Kemp and Speicher \cite{KS} have generalized Haagerup's inequality for $\mathcal R$-diagonal 
elements in the so-called holomorphic case 
(with polynomials in the variables, but not their adjoint). 
Theorem \ref{MainTh} established, 
the consequence for random matrices sounds relevant since it allows to consider combinations of 
Haar and deterministic matrices, 
and then get a bound for its operator norm. 
The result of \cite{KS} we state below has been generalized by 
de la Salle \cite{SAL} in the case where the non commutative random variables have matrix coefficients. 
This situation could be interesting for practical applications, where block random matrices are sometimes 
considered (see \cite{TV} for applications of random matrices in telecommunication). 
Nevertheless, we only consider the scalar version for simplicity.
\\
\\Recall that a non commutative random variable $a$ is called an $\mathcal R$-diagonal element if it 
can be written $a = u y$, for $u$ a Haar unitary free from $y$ (see \cite{NS}). 
Let $\mathbf a = (a_1 \etc a_p)$ be a family of free, identically distributed $\mathcal R$-diagonal 
elements in a $\mathcal C^*$-probability space $(\mathcal A, .^*, \tau, \| \cdot \| )$. 
We denote by $W_d^+$ the set of reduced monomials of length $d$ in variables $\mathbf x$ (and not its conjugate), i.e.
		$$W_d^+ =  \Big\{ \  x_{j_1} \dots  x_{j_d} \  \Big | \  j_1 \neq \dots \neq j_d \ \Big \}.$$
Kemp and Speicher have shown the following, where the interesting fact is that the 
constant $(d+1)$ is replace by a constant of order $\sqrt{d+1}$: for any sequence $(P_n)_{n\geq 1} $ 
of elements of $W_d^+$ and any sequence $\alpha = (\alpha_n)_{n\geq 1}$, one has
\begin{equation}\label{eq:HaaIneq2}
		\Big \| \sum_{n\geq 1} \alpha_n P_n(\mathbf a) \Big \| \leq e \sqrt{d+1} \Big \| \sum_{n\geq 1} \alpha_n P_n(\mathbf a)Ê\Big \|_2,
\end{equation}
where $\| \cdot \|_2$ denotes the $L^2$-norm in $\mathcal A$, given by 
$\|x\|_2 = \tau[x^*x]^{1/2}$ for any $a$ in $\mathcal A$. In particular, if $\mathbf a = \mathbf u$ is a family of free unitaries
(i.e. $y=\mathbf 1$) then we get $\| \sum_{n\geq 1} \alpha_n P_n(\mathbf u) \Big \|_2 = \| \alpha \|_2$, so that 
(\ref{eq:HaaIneq2}) is already an improvement of (\ref{eq:HaaIneq}) without the 
generalization on $\mathcal R$-diagonal elements.
\\
\\Now let $\mathbf U_N = (U_1\toN \etc U_p\toN), \mathbf V_N =(V_1\toN \etc V_p\toN)$ be families of 
$N \times N$ independent unitary Haar matrices and $\mathbf Y_N=(Y_1\toN \etc Y_p\toN)$ 
be a family of $N \times N$ deterministic Hermitian matrices. Assume that for any $j=1\etc p$, the empirical 
spectral distribution of $Y_j\toN$ converges weakly to a measure $\mu$ (that does not depend on $j$) and 
that for $N$ large enough, the eigenvalues of $Y_j\toN$ belong to a small neighborhood of the support of 
$\mu$. We set for any $j=1\etc p$ the random matrix 
		$$A_j\toN = U_j\toN  Y_j\toN V_j\toNs.$$
From Theorem \ref{MainTh} and \cite[Corollary 2.1]{MAL}, we can deduce that almost surely the family 
$\mathbf A_N = (A_1\toN \etc A_p\toN )$ converges strongly in law to a family $\mathbf a$ of free $\mathcal R$-diagonal elements, identically distributed. Hence, inequality (\ref{eq:HaaIneq2}) gives: for any polynomials $P_n$ in $W_d^+$, $n\geq 1$,
		$$\lim_{N \rightarrow \infty} \| \sum_{n\geq 1} \alpha_n P_n(A_1\toN \etc A_p\toN) \|   \leq e \sqrt{d+1}\Big \| \sum_{n\geq 1} \alpha_n P_n(\mathbf a)Ê\Big \|_2 .$$

\section{Proof of Theorem \ref{MainTh}}
\label{sec:proof}

\noindent We consider a unitary Haar matrix $U_N$, 
independent of
a family of matrices $\mathbf Y_N$, 
having almost surely a strong limit in distribution. We show that almost surely $(U_N,\mathbf Y_N)$ has almost surely a strong limit in distribution. As it is known that $(U_N,\mathbf Y_N)$ 
converges in distribution \cite[Theorem 5.4.10]{AGZ}, the only thing we have to show 
is the convergence of norms. This will show Theorem \ref{MainTh} by recurrence on the number 
of Haar matrices. Moreover, the problem can be simplified in the following way (see \cite[Section 3]{MAL}):
\begin{itemize}
	\item one can reason conditionally, and then assume that the matrices of $\mathbf Y_N$ are deterministic,
	\item one may assume that the matrices of $\mathbf Y_N$ are Hermitian by considering their Hermitian and anti-Hermitian parts,
	\item it is sufficient to prove that for any polynomial $P$, almost surely the norm of $\big\| P(\mathbf U_N, \mathbf U_N^*, \mathbf Y_N)\big\|$ converges, rather than ``almost surely, for any polynomial''.
\end{itemize}

\noindent The keystone of the proof is the use of a classical coupling of real random variables, namely the inverse transform sampling method, for Hermitian matrices (Lemma \ref{Coupling} below). It allows us to get the strong convergence of $(U_N,\mathbf Y_N)$ from the strong convergence of $(X_N,\mathbf Y_N)$, where $X_N$ is a GUE matrix independent of $\mathbf Y_N$, for which we know the strong convergence by Theorem \ref{thm:CM}. For that purpose, we will first go through an intermediate problem. We use the coupling to prove in Lemma \ref{Step1} the strong convergence of $(M_N, \mathbf Y_N)$, where $M_N$ is the unitary invariant random matrix whose spectrum is $\big\{ \frac 1 N  \etc   \frac {N-1}N ,  \frac N N \big\}$. 
From this, we  deduce that the strong convergence holds for $(Z_N, \mathbf Y_N)$, where $Z_N$ is any unitary invariant random matrix, independent of $\mathbf Y_N$, whose spectrum is $\big\{ \gamma_N\big(\frac 1 N\big)  \etc \gamma_N\big (\frac {N-1}N \big) ,  \gamma_N\big( \frac N N\big) \big\}$ for $\gamma_N:[0,1] \to \mathbb C$ a random map converging uniformly to a continuous map. We finally use the coupling method again, to remark that a unitary Haar matrix could be written as above.

\subsection{Preliminaries}

\noindent Let $a$ be a self-adjoint element in a $\mathcal C^*$-algebra, that is $a^*=a$. Denote by $\mu_a$ its spectral distribution, i.e. $\mu_a$ is the unique probability measure on $\mathbb R$ such that for any $k\geq 1$, $\int t^k \textrm d\mu_a(t) = \tau[a^k].$ This measure has compact support included in $\big[-\|a\|, \|a\| \big]$. Denote by $F_a$ its cumulative function, satisfying, for all $t$ in $\mathbb R$:
	\begin{equation}
		F_a(t) := \mu \big( -]\infty, t] \big ).
	\end{equation}
We set the generalized inverse of $F_a$: for any $s$ in $]0,1]$
	\begin{equation}
		F^{-1}_{a}(s) = \inf\big \{ t \in [-\pi,\pi] \ \big | \ F_{a}(t) \geq s \big\}.
	\end{equation}

\noindent By the inverse method for random variables \cite[Chapter two]{DEV86}, we get the following lemma.
\\
\begin{Lem}[The coupling of self-adjoint variables and Hermitian matrices by cumulative functions]~\label{Coupling}
\begin{enumerate}
	\item Let $a,b$ be two self-adjoint non commutative variables. Denote the self-adjoint variables $\tilde b = F_b^{-1} \circ  F_a(a)$ given by functional calculus. If $\mu_a$ have no discrete part (i.e. $\mu_a\big(\{t\}\big) =0$ for any $t$ in $\mathbb R$), then $\tilde b$ has the same distribution as $b$.
	\item Let $A_N$ and $B_N$ be two Hermitian matrices (living in the $\mathcal C^*$-probability space \break $(\MN, .^*, \tau_N, \| \cdot \|)$). Write the matrices $A_N = V_{A_N} \ \Delta_{A_N} \ V_{A_N}^*$ and $ B_N = V_{B_N} \ \Delta_{B_N} \ V_{B_N}^*$, with $V_{A_N}, V_{B_N}$ unitary matrices, such that the entries of $\Delta_{A_N}$ and $\Delta_{B_N}$ are non decreasing along the diagonal. Assume that diagonal entries of $\Delta_{A_N}$ are distinct. We set the matrix
	\begin{equation}
		M_N := V_{A_N} \ \textrm{diag }\Big( \frac 1 N  \etc   \frac {N-1}N ,  \frac N N  \Big) \  V_{A_N}^*.
	\end{equation}
	Then $M_N = F_{A_N}(A_N)$ and $F^{-1}_{B_N}(M_N) = V_{A_N} \ \Delta_{B_N} \ V_{A_N}^*$.
\end{enumerate}
\end{Lem}

\subsection{Step 1: from the GUE to an intermediate model} 

\begin{Lem}[The strong asymptotic freeness of $M_N, \mathbf Y_N$] \label{Step1} ~
\\Define the random matrix 
	\begin{equation}
		M_N = V_N \ \textrm{diag }\Big( \frac 1 N  \etc   \frac {N-1}N ,  \frac N N  \Big) \ V_N^*,
	\end{equation}
where $V_N$ is a unitary Haar matrix, independent of $\mathbf Y_N$. Then, almost surely $(M_N, \mathbf Y_N)$ converges strongly in distribution to $(m, \mathbf y)$, where $\mathbf y$ is the strong limit of $\mathbf Y_N$, free from a self adjoint variable $m$ whose spectral distribution is the uniform measure on $[0,1]$.
\end{Lem}

\begin{proof} Let $X_N$ be a GUE matrix independent from $\mathbf Y_N$, such that $X_N = V_N \Delta_N V_N^*$, where $\Delta_N$ is a diagonal matrix, independent of $V_N$, with non decreasing entries along the diagonal (we recall a proof of that decomposition 
in Proposition \ref{prop:SpDec}, Section \ref{sec:SpDec}). Let $x$ be a semicircular variable free from the strong limit $\mathbf y$ of $\mathbf Y_N$. Let $F_{X_N}$ and $F_x$ be the cumulative functions of the spectral measures of $X_N$ and $x$ respectively. By the coupling of Lemma \ref{Coupling}, we get that $m = F_{x}(x)$ has the expected distribution and, since the eigenvalues of a GUE matrix are almost surely distinct, we get that almost surely $M_N = F_{X_N}(X_N)$. Then, for any polynomial $P$, almost surely 
	\begin{eqnarray} \label{eq:Step1}
		 \lefteqn{ \Big|  \|  P(m, \mathbf y)    \| - \| P(M_N, \mathbf Y_N) \| \Big|
		 	}\\
		& \leq & \Big| \| P( F_x(x), \mathbf y) \|  -  \| P( F_x(X_N), \mathbf Y_N)  \| \Big| \nonumber
		\\
		 &  & + \  \|  P( F_x(X_N), \mathbf Y_N) - P( F_{X_N}(X_N), \mathbf Y_N) \|. \nonumber
\end{eqnarray}

\noindent The first term in the right hand side of \eqref{eq:Step1} tends to zero almost surely by the strong convergence in distribution of $(X_N, \mathbf Y_N)$ to $(x, \mathbf y)$ (Theorem \ref{thm:CM}) and Proposition \ref{Prop:ConvEqui} since $F_{x}$ is continuous. For the second term, recall first that the convergence in distribution of $X_N$ to $x$ implies the pointwise convergence $F_{X_N}$ to $F_x$ (at any point of continuity of $F_x$, and so on $\mathbb R$). By Dini's theorem \cite[Problem 127 chapter II]{PS98}, $F_{X_N}$ convergence actually uniformly to $F_x$. Hence, since the matrices $X_N, \mathbf Y_N$ are uniformly bounded in operator norm and by local uniform continuity of $P$, the second term converges also to zero. 
\end{proof}

\subsection{Step 2: from the reference model to other unitary invariant models}

\begin{Lem}[The strong asymptotic freeness of $Z_N, \mathbf Y_N$] \label{Step2} ~
\\Consider an $N$ by $N$ random matrix $Z_N$ of the form
	\begin{equation}
		Z_N := \gamma_N(M_N) = V_N \ \textrm{diag }\bigg( \gamma_N\Big( \frac 1 N\Big)  \etc \gamma_N \Big(  \frac {N-1}N \Big), \gamma_N \Big( \frac N N  \Big) \bigg) \ V_N^*,
	\end{equation}

\noindent where $M_N$ is the random matrix of Lemma \ref{Step1} and $\gamma_N : [0,1] \to \mathbb C$ is a random map, independent of $\mathbf Y_N$. Assume that almost surely $\gamma_N$ converges uniformly to a continuous map $\gamma: [0,1] \to \mathbb C$, that is
			$$ \|\gamma - \gamma_N \|_\infty := \sup_{t\in [0,1]}  |\gamma(t) - \gamma_N(t) | \limN 0.$$ 

\noindent We set the self adjoint variable $z = \gamma(m)$ given by functional calculus, with $m$ as in Lemma \ref{Step1} (it is well defined since $\|m\|\leq 1$). Then, almost surely, $(Z_N,\mathbf Y_N)$ converges strongly to $(z, \mathbf y)$.
\end{Lem}

\begin{proof} For any polynomial $P$, one has
	\begin{eqnarray} \label{eq:Step2}
		 \lefteqn{ \Big|  \|  P(z, z^*, \mathbf y)    \| - \| P(Z_N, Z_N^* \mathbf Y_N) \| \Big|
		 	}\\
		& \leq & \Big| \big\| P\big( \gamma(m), \bar \gamma(m), \mathbf y\big) \big\|  -  \big \| P \big(\gamma(M_N), \bar \gamma(M_N), \mathbf Y_N \big)  \big\| \Big| \noindent
		\\
		 &  & + \  \big\|  P\big( \gamma(M_N), \bar \gamma(M_N), \mathbf Y_N\big) - P\big(\gamma_N(M_N), \bar \gamma_N(M_N), \mathbf Y_N\big) \big\|, \noindent
\end{eqnarray}

\noindent where $\bar \gamma$ denotes the complex conjugacy of $\gamma$. The first term of the right hand side of \eqref{eq:Step2} tends to zero by Lemma \ref{Step1}, Proposition \ref{Prop:ConvEqui}, and the continuity of $\gamma$. By the uniform convergence of $\gamma_N$, the continuity of polynomials, and the fact that the matrices $M_N, \mathbf Y_N$ are uniformly bounded in operator norm, the second term vanishes at infinity.
\end{proof}

\subsection{Step 3: application to Haar matrices}

\noindent Now, let $U_N$ be a unitary Haar matrix, independent of $\mathbf Y_N$. By the spectral theorem (see Proposition \ref{prop:SpDec} in Section \ref{sec:SpDec}), we can write $U_N = V_N \Delta_N V_N^*$, where the entries of $\Delta_N = \textrm{diag }(e^{2\pi i\theta_1} \etc e^{2\pi i\theta_N})$ have non decreasing argument in $[0,2\pi[$ along the diagonal. Define as above $M_N = V_N \ \textrm{diag }( \frac 1 N \etc \frac {N-1}N, \frac N N) \ V_N^*$, where $V_N$ is the unitary matrix in the decomposition of $U_N$. Denote by $F_{U_N}$ the cumulative function of $ \textrm{diag }( \theta_1\toN \etc \theta\toN_N)$. We get by the coupling of Lemma \ref{Coupling} that
	\begin{equation}
		U_N = \exp\big( 2 \pi i F_{U_N}^{-1} (M_N) \big).
	\end{equation}

\noindent By Lemma \ref{Step2}, to get the strong convergence of $(U_N, \mathbf Y_N)$ it remains to prove that almost surely $\gamma_N : t \to  \exp \big( 2 \pi i F_{U_N}^{-1} (t) \big)$ converges uniformly to $\gamma : t \to  \exp \big( 2 \pi i t \big)$.
\\
\\From the convergence of $U_N$ to a Haar unitary $u$, we get that almost surely $F_{U_N}$ converges to $F_u$. Let $t$ in $[0,1[$. Almost surely, for any $0<\alpha<1-t$, there exists $N_0 \geq 1$ such that for any $N\geq N_0$, $F_{U_N}(t+\alpha) \geq t + \frac \alpha 2$. The points $\big(F^{-1}_{U_N}(t), t\big)$ and $\big( t+ \alpha, F_{U_N}(t+\alpha) \big)$ belong to the graph of $F_{U_N}$, with vertical segments on points of discontinuity. Hence, since $F_{U_N}$ is non decreasing we get $F^{-1}_{U_N}(t) \leq t + \alpha$. Hence, for any $0\leq t< 1$, we get $\limsup_{N\rightarrow \infty} F^{-1}_{U_N}(t) \leq t$. With a symmetric reasoning we get $\liminf_{N\rightarrow \infty} F^{-1}_{U_N}(t) \geq t$. Now, remark that $F^{-1}_{U_N}(0)= \theta_1 \geq 0$ and $F^{-1}_{U_N}(1)= \theta_N \leq 1$. Hence, $F^{-1}_{U_N}$ converges pointwise to the identity map on $[0,1]$. By Dini's theorem, it converges uniformly. Hence $\gamma_N$ converges uniformly to $\gamma$.

\section{Proof of Theorem \ref{thm:COE}}\label{proofCOE}

\noindent The proof of \ref{thm:COE} is obtained by changing the words unitary, Hermitian and GUE into orthonormal, symmetric and GOE, respectively symplectic, self dual and GSE, and by taking $\mathbf Y_N$ a family of independent orthogonal, respectively symplectic, matrices. Instead of Theorem \ref{thm:CM} we use the main result of \cite{SCH}. In the symplectic case, we also have to consider matrices of even size.

\section{Proof of Corollary \ref{cor:SumProd}} \label{sec:ProofCorSumProd}

\noindent First recall the following consequence of \cite[Corollary 2.1]{MAL}.

\begin{Lem}\label{lem:DiagMatrices} Let $D_1\toN$ and $D_2\toN$ be two diagonal matrices having a strong limit in distribution separately. Then, there exists diagonal matrices $\tilde D_1\toN$ and $\tilde D_2\toN$, with the same eigenvalues as $D_1\toN$ and $D_2\toN$ respectively, such that $(\tilde D_1\toN, \tilde D_2\toN)$ converges strongly in distribution.
\end{Lem}

\noindent Let $A_N$ and $B_N$ be as in Corollary \ref{cor:SumProd}. Without loss of generality, we can assume that the laws of $A_N$ and $B_N$ are invariant under unitary conjugacy. Let $\Delta_{A_N}$ and $\Delta_{B_N}$ be diagonal matrices of eigenvalues of $A_N$ and $B_N$ respectively. By Proposition \ref{Prop:ConvEqui}, the assumptions on $A_N$ and $B_N$ mean their strong convergence in distribution separately, and so the strong convergence of $\Delta_{A_N}$ and $\Delta_{B_N}$ separately. With the notations of Lemma \ref{lem:DiagMatrices}, consider $\tilde \Delta_{A_N}$ and $\tilde \Delta_{B_N}$. Let $(U_N, V_N)$ be independent unitary Haar matrices, independent of $(\tilde \Delta_{A_N}, \tilde \Delta_{B_N})$. Then $(A_N, B_N)$ and $(U_N \tilde \Delta_{A_N} U_N^*, V_N \tilde \Delta_{B_N} V_N^*)$ are pairs of random matrices with the same probability law (see Proposition \ref{prop:SpDec}). By Theorem \ref{MainTh}, we get the almost sure strong convergence of $(U_N,V_N, \tilde \Delta_{A_N} , \tilde \Delta_{B_N})$, and then of $(U_N \tilde \Delta_{A_N} U_N^*, V_N \tilde \Delta_{B_N} V_N^*)$. Hence, we obtain that $(A_N, B_N)$ has a strong limit in distribution $(a,b)$. The spectral distribution of $a$ is $\mu$, the one of $b$ is $\nu$, and $a$ and $b$ are free. The strong convergence implies the convergence of the spectrum of $A_N + B_N$ to the support of $\mu \boxplus \nu$ (which is the spectral distribution of $a+b$) by Proposition \ref{Prop:ConvEqui}. We then get the first point of Corollary \ref{cor:SumProd}.
\\
\\We get the second point of Corollary \ref{cor:SumProd} with the same reasoning on 
$( \Delta_{A_N}, \Delta_{B_N}^{1/2})$. The application stated after Corollary \ref{cor:SumProd} 
follows by taking $\Pi_N=B_N$, which satisfies the assumptions since $t\in (0,1)$, and remarking 
that $\Pi_N^{1/2}=  \Pi_N$.

\section{Appendix: The spectral theorem for unitary invariant random matrices}\label{sec:SpDec}

\noindent This result seems to be folklore in the literature of Random Matrix Theory, but we are
not able to find an exact reference, so we include a proof for the convenience of the readers. 

\begin{Prop}[The spectral theorem for unitary invariant random matrices]\label{prop:SpDec}
Let $M_N$ be an $N \times N$ Hermitian or unitary random matrix whose distribution is invariant under conjugacy by unitary matrices. Then, $M_{N}$ can be written $M_N = V_N \Delta_N V_N^*$ almost surely, where
\begin{itemize}

\item $V_N$ is distributed according to the Haar measure on the unitary group,

\item $\Delta_N$ is the diagonal matrix of the eigenvalues of $M_N$, arranged in increasing order if $M_N$ Hermitian, and in increasing order with respect to the set of arguments in $[-\pi,\pi[$ if $M_N$ is unitary,
\item $V_N$ and $\Delta_N$ are independent.
\end{itemize}
\end{Prop}

\noindent We actually use the proposition only for unitary Haar and GUE matrices, 
which are two cases where almost surely the eigenvalues are distinct. 
The fact that multiplicities of eigenvalues is almost surely one brings slight conceptual simplifications in 
the proof, but nevertheless do not change the result. Hence, we choose to state the proposition without any restriction on the multiplicity of the matrices.

\begin{proof}
By reasoning conditionally, one can always assume that the multiplicities of the eigenvalues of $M_N$ 
are 
almost surely constant. We denote by $(N_1 \etc N_K)$ the sequence of multiplicities when the eigenvalues 
are considered in the natural order in $\mathbb R$ or in increasing order with respect to their argument in $[-\pi,\pi[$.
\\
\\Since 
$M_N$ is normal, it can be written $M_N = \tilde V_N  \Delta_N  \tilde V_N$, where 
$\tilde V_N$ is a random unitary matrix and $\Delta_N$ is as announced. The choice of $\tilde V_N$ can 
be made in a measurable way, see for instance \cite[Section 5.3]{DEI}, with minor modifications.
\\
\\Let $(u_1\etc u_K)$ be a family of independent random matrices, independent of $(\Delta_N,\tilde V_N)$ and 
such that for any $k=1\etc K$, the matrix $u_k$ is distributed according to the Haar measure on 
$\mathcal U(N_k)$, the group of $N_k \times N_k$ unitary matrices. We set 
		$$V_N=\tilde V_N \ \mathrm{diag} \ (u_1\etc u_K),$$
and claim that the law of $V_N$ depends only on the law of $M_N$, not in the choice of the random matrix 
$\tilde V_N$. Indeed, let $M_N = \bar V_N \Delta_N \bar V_N$ be an other decomposition, where 
$\bar V_N$ is a unitary random matrix, independent of $(u_1\etc u_K)$. The multiplicities of the 
eigenvalues being $N_1 \etc N_K$, there exists $(v_1\etc v_K)$ in 
$\mathcal U(N_1)\times \dots \times \mathcal U(N_K)$, independent of $(u_1\etc u_K)$, such that 
$\bar V_N = \tilde V_N \ \mathrm{diag} \ (v_1\etc v_K)$. Hence, we get 
$\bar V_N \ \mathrm{diag} \ (u_1\etc u_K) = \tilde V_N \ \mathrm{diag} \ (v_1u_1\etc v_Ku_K)$, which is equal in law to $V_N$. This proves the claim.
\\
\\Let $W_N$ be an $N \times N$ unitary matrix. Then 
$W_N M_N W_N^* = (W_N \tilde V_N) \Delta_N (W_N \tilde V_N)^*$. 
By the above, since $M_N$ and $W_N M_N W_N^* $ are equal in law, 
then $V_N$ and $W_N V_N$ are also equal in law. Hence $V_N$ is Haar distributed in $\mathcal U(N)$.
\\
\\It remains to show the independence between $V_N$ and $\Delta_N$. Let $f : \mathcal U(N) \to \mathbb C$ 
and $g : \MN \to \mathbb C$ two bounded measurable functions such that $g$ depends only on the eigenvalues 
of its entries. Then one as $\esp \big[ f(V_N) g(\Delta_N) \big]  =  \esp \big[ f(V_N) g(M_N) \big]$. Let $W_N$ 
be Haar distributed in $\mathcal U(N)$, independent of $(V_N,\Delta_N)$. Then by the invariance under 
unitary conjugacy of the law of $M_N$, one has
\begin{eqnarray*}
	\esp \big[ f(V_N) g(\Delta_N) \big] & = &  \esp \big[ f(W_N V_N) g(W_N M_N W_N^*) \big] \\
		& = &  \esp \big[ f(W_N V_N) g(\Delta_N) \big]\\
		& = &  \esp \Big[ \esp \big [ f(W_N V_N) \big | V_N,\Delta_N \big ] g(\Delta_N) \big]\\
		& = &  \esp \big[ f(W_N) \big] \esp \big[ g(\Delta_N) \big] =\esp \big[ f(V_N) \big] \esp \big[ g(\Delta_N) \big].
\end{eqnarray*}\end{proof}

\section{Acknowledgments}

\noindent B.C. would like to thank Gilles Pisier for asking him Question \eqref{eq:QuestionPisier} 
many years ago, as it was a source of 
inspiration for this paper.
C.M. would like thank his Ph.D. supervisor Alice Guionnet for submitting him as a part of his Ph.D. project 
a question that lead to the main theorem of this paper and for her insightful guidance. 
He would also like to thank Ofer Zeitouni and Mikael de la Salle for useful discussions during the preparation of this article.

This paper was initiated and mostly completed during the 
Erwin Schr\"odinger International Institute for Mathematical Physics
workshop 
``Bialgebras in Free Probability'' in April 2011.
Both authors gratefully acknowledge the ESI and the organizers of the meeting
for providing an inspiring working environment.

The authors would like to thank the anonymous referee(s) for their constructive suggestions of improvements of
the manuscript.

Both authors acknowledge the financial support of the ESI. 
The research of B.C. was supported in part by a
Discovery grant from the Natural Science and Engineering Research Council of Canada, an Early Research Award of the Government of Ontario,
and the ANR GranMa.

\bibliographystyle{plain} 
\bibliography{biblio}
\end{document}